\documentclass{amsart}
\usepackage{amssymb,amsrefs,amsmath,amsthm,graphicx,enumerate}
\usepackage[all,cmtip]{xy}
\usepackage{hyperref}

\usepackage{geometry}
\newtheorem{thm}{Theorem}[section]
\newtheorem{prop}[thm]{Proposition}
\newtheorem{lem}[thm]{Lemma}
\newtheorem{cor}[thm]{Corollary}
\newtheorem{con}[thm]{Conjecture}
\theoremstyle{definition}
\newtheorem{definition}[thm]{Definition}
\newtheorem{example}[thm]{Example}
\theoremstyle{remark}
\newtheorem{remark}[thm]{Remark}
\numberwithin{equation}{section}

\newcommand{\Z}{\mathbb{Z}}  
\newcommand{\Q}{\mathbb{Q}}  
\newcommand{\C}{\mathbb{C}}  
\newcommand{\A}{\mathcal{A}}
\newcommand{\CP}{\mathbb{CP}^{n+1}}  
\newcommand{\TA}{H^{\varepsilon,\rho}} 
\newcommand{\TC}{C^{\varepsilon,\rho}} 
\newcommand{\Ft}{\mathbb{F}[t^{\pm 1}]} 
\newcommand{\Qt}{\mathbb{Q}[t^{\pm 1}]} 
\newcommand{\Ct}{\mathbb{C}[t^{\pm 1}]} 

\begin{document}

\title{Twisted Alexander polynomials of hypersurface complements}

\author{KaiHo Tommy Wong}
\address{Department of Mathematics, University of Wisconsin at Madison, 
Madison, WI 53706}
\email{wong@math.wisc.edu}
\urladdr{www.math.wisc.edu/$\sim$wong}

\begin{abstract}
We prove that the acyclity assumption required in Cogolludo and Florens work on twisted Alexander polynomials of plane curve complements, are generically satisfied. We then define twisted Alexander polynomials of a complex hypersurface with arbitrary singularities. These generalize the classical Alexander polynomials of high dimensional hypersurfaces and the twisted Alexander polynomial of plane curves. We recover the classical torsionness and divisibility results, which say that, under certain assumptions, the twisted Alexander modules of a complex hypersurface are torsion modules, and that their orders divide the product of certain `local polynomials' defined in terms of the topology at singularities.
\end{abstract}

\maketitle
\section{Introduction}
The Alexander invariant was first defined by considering the first homology of the infinite cyclic cover of the knot or link exterior as a module over the Laurent polynomial ring. It is a strong and useful knot invariant both in practice and in theory (\cite{RB}). 

Since any germ of a complex plane curve is given by a link pair, Libgober adopted the Alexander invariant to the study of affine plane curve complements (\cite{AC}). He proved that the Alexander polynomial of a plane curve complement divides the product of the Alexander polynomials of the link exteriors associated to the singular points. Moreover, if the curve is transversal at infinity, then its Alexander polynomial also divides the Alexander polynomial of the link pair at infinity given by the formula $$(t-1)(t^d-1)^{d-2},$$ where $d$ is the degree of the curve.

This was generalized to the case of complex hypersurface complements with only isolated singularities by Libgober himself (\cite{HG}), and with non-isolated singularities by Maxim (\cite{MT}), and by Dimca and Libgober (\cite{RF}). Recently in 2014, Liu also recovered similar divisibility results by using nearby cycles (\cite{LT}).

Cogolludo and Florens consider Alexander polynomials twisted by linear representations (\cite{TC}), which are useful in the study of 3-manifolds. They improve Libgober's divisibility result by providing an equation, which will be referred to as the Cogolludo-Florens equation in this paper, and give some applications of twisted Alexander polynomials. Their work assumes acyclicity of certain twisted chain complexes.

In this paper, we will further investigate the twisted Alexander modules, and in particular, the terms in the equation given in Cogolludo and Florens's paper. Provided that a plane curve is transversal at infinity, the twisted Alexander modules for the link pair at infinity, and certain local twisted Alexander modules at singular points, are indeed torison. As a result, in many cases, the first twisted Alexander module is torsion, giving us a well-defined polynomial invariant. Also, in many interesting cases, the acyclicity assumption in Cogolludo and Florens's paper can be discarded.

We also study twisted Alexander invariants for complex hypersurface complements with possibly non-isolated singularities, focusing on the ones associated to the total linking number homomorphism. We will show that these modules are again torsion over the Laurent polynomial ring in some range, depending on the dimension of the ambient spaces. Divisibility results similar to those in the classical setting will be proved using some classical topology of stratified spaces.

Finally, it is known that roots of classical Alexander polynomials are the roots of unity. We prove a similar fact for twisted Alexander polynomials.
\vspace{0.5cm}

\textbf{Acknowledgment} I would like to thank my advisor, Laurentiu Maxim, for his motivation and support in this project. I am also very grateful to Anatoly Libgober, Alexandru Dimca, Yongqiang Liu, and Yun Su for useful discussions.

\section{Twisted chain complexes and Twisted Alexander Modules}

\subsection{Definitions}

We recall the definitions of twisted chain complexes, twisted Alexander modules, and twisted Alexander polynomials of path connected finite CW-complexes (\cite{TC},\cite{KL}).

Let $X$ be a path-connected finite CW-complex with $\pi=\pi_1(X)$. Assume there is a group homomorphism $$\varepsilon: \pi_1(X) \rightarrow \Z.$$ Let $\mathbb{F}$ be a field and consider a finite dimensional $\mathbb{F}$-vector space $\mathbb{V}$ and a linear representation $$\rho: \pi \rightarrow GL(\mathbb{V}).$$ Note that $\varepsilon$ extends to an algebra homomorphism $$\varepsilon: \mathbb{F}[\pi] \rightarrow \mathbb{F}[\Z] \cong \Ft.$$

Let $\tilde{X}$ be the universal cover of $X$. The cellular chain complex $C_*(\tilde{X};\mathbb{F})$ is a complex of left $\mathbb{F}[\pi]$-modules, generated by the lifts of the cells of $X$. Take $\mathbb{V}$ as a 2-sided $\mathbb{F}[\pi]$-module, with action for $v\in \mathbb{V}$ and $\alpha\in \pi$, $$v\cdot \alpha = (\rho(\alpha))(v)$$ and $$\alpha\cdot v= (\rho(\alpha))^{-1}(v).$$

Also consider the right $\mathbb{F}[\pi]$-module $\Ft \otimes \mathbb{V}$, where the action is induced by $\varepsilon\otimes \rho$:
$$(p\otimes v)\cdot \alpha = pt^{\varepsilon(\alpha)} \otimes v\cdot\alpha= pt^{\varepsilon(\alpha)} \otimes \rho(\alpha)v, \alpha \in \pi.$$

Let the chain complex of $(X,\varepsilon,\rho)$ be defined as the complex of $\Ft$-modules:

$$\TC_*(X,\Ft) = (\Ft \otimes \mathbb{V}) \otimes_{\mathbb{F}[\pi]}C_*(\tilde{X};\mathbb{F}),$$ where the action is given by $$t^n ((p\otimes v)\cdot c) = (t^n \cdot p \otimes v)\cdot c.$$

It is complex of free modules. A basis is given by elements of the form $1\otimes e_i \otimes c_k$, where $\{e_i\}$ is a basis of $\mathbb{V}$ and $\{c_k\}$ is a basis of the $\mathbb{F}[\pi]$ module $C_*(\tilde{X};\mathbb{F})$, obtained by lifting cells of $X$.

\begin{definition}
The twisted Alexander module $\TA_*(X,\Ft)$ of the triple $(X, \varepsilon, \rho)$ over the group ring $\Ft$ is defined to be the homology of the twisted chain complex $H_*(\TC_*(X,\Ft))$ with the induced $\Ft$-action.
\end{definition}

These modules are homotopy invariants. Theorem 2.1 from \cite{KL} gives the following equivalent definition of the twisted complex of $(X,\varepsilon,\rho)$. Suppose $X_{\infty}$ is the infinite cyclic cover of $X$ associated to $\pi'$, where $\pi'$ is the kernel of $\varepsilon$. Then the chain complex $$C_*(X_{\infty};\mathbb{V}_{\rho}):=\mathbb{V}\otimes_{\mathbb{F}[\pi']} C_*(\tilde{X}),$$ considered as a $\Ft$-module, is isomorphic to $\TC_*(X,\Ft)$. The action in the module $C_*(X_{\infty};\mathbb{V}_{\rho})$ is given by $t^n \cdot (v\otimes c) = v\cdot\gamma^{-n} \otimes \gamma^n c$, where $\gamma$ is an element in $\pi$ such that $\varepsilon(\gamma)=1$.

Note that if $\varepsilon$ is surjective, then $X_{\infty}$ is connected. Otherwise, there is a bijection between the set of path connected component of $X_{\infty}$ and cardinality of the cokernel of $\varepsilon$.

Denote by $\mathbb{F}(t)$ the field of fractions of $\Ft$, and define $$\TC_*(X,\mathbb{F}(t)) = \TC_*(X;\Ft) \otimes \mathbb{F}(t).$$ $(X,\varepsilon,\rho)$ is called acyclic if the chain complex $\TC_*(X,\mathbb{F}(t))$ is acyclic over $\mathbb{F}(t)$. Since $\Ft$ is a principal ideal domain, $\mathbb{F}(t)$ is flat over $\Ft$. So, $(X,\varepsilon,\rho)$ is acyclic if and only if $\TA_*(X,\Ft)$ are torsion over $\Ft$. Because $\Ft$ is a principal ideal domain, $\TA_*(X,\Ft)$ has a decomposition of free part and torsion part.

\begin{definition}
The order of the torsion part of $\TA_i(X,\Ft)$ is called the $i$-th twisted Alexander polynomial of $(X,\varepsilon,\rho)$, and is denoted by $\Delta_{i,X}^{\varepsilon,\rho}(t)$. 
\end{definition} 

These twisted Alexander polynomials are defined up to units in $\Ft$.

The following proposition says that in the `usual' situation, for instance, if the infinite cyclic cover $X_{\infty}$ is connected, then the 0-th twisted Alexander module is torsion.

\begin{prop}\cite{KL} If $\varepsilon$ is non-trivial, then $\TA_0(X,\Ft)$ is torsion over $\Ft$.
\end{prop}

See \cite{WD} for examples in the case of knots. Here are two other important examples.

\subsection{Generalized Hopf link}

Let $K$ be the generalized Hopf link in $S^3$, that is, the link with $d\geq 2$ components and linking number one for each pair of components. We calculate the twisted Alexander polynomials of the link exterior $S^3\setminus K$.

\begin{lem}
If $V\subset \C^n$ is a reduced hypersurface defined by a homogeneous polynomial and having a hyperplane as one of its components, then the restriction of the Hopf bundle $$\C^n\setminus V \rightarrow \mathbb{P}^{n-1}\setminus [V]$$ is trivial.
\end{lem}

\begin{proof}
Let $H$ be a hyperplane component of $V$. Then the Hopf bundle $$\C^n\setminus H \rightarrow \mathbb{P}^{n-1}\setminus [H] \cong \C^{n-1}$$ is trivial, in particular any of its restrictions are trivial.
\end{proof}

\begin{prop}
$$\pi_1(S^3\setminus K)\cong  \Z \times \mathbb{F}_{d-1} \cong <x_0,x_1,...,x_{d-1}| x_0x_ix_0^{-1}x_i^{-1}, i=1,...,d-1>.$$
\end{prop}

\begin{proof}
There are two proofs. Both use the fact (\cite{AC}), that $S^3 \setminus K$ is homotopy equivalent to the link exterior of the singularity $x^d=y^d \subset \C^2$. Let $W:=\C^2 \setminus \A$, where $\A$ is the central line arrangement of $d$ lines. Then $S^3\setminus K \simeq \C^2 \setminus \A$.

By p.117 in \cite{DB}, the fundamental group is computed directly using Van-Kampen theorem, 
$$\pi_1(S^3\setminus K) \cong \pi_1(W)  = <x_0,x_1,...,x_d| x_dx_{d-1}\cdots x_1 x_0^{-1}, x_0x_ix_0^{-1}x_i^{-1}, i=1,...,d>.$$

So, $\pi_1(S^3\setminus K) \cong <x_0,x_1,...,x_{d-1}| x_0x_ix_0^{-1}x_i^{-1}, i=1,...,d-1>$.

The second proof uses the previous lemma on $W$. We have $$W \cong \C^*\times (\mathbb{P}^1\setminus \{\text{ d points}\}).$$

Hence, $S^3\setminus K \simeq S^1\times( \vee^{d-1} S^1)$, which gives the desired presentation of the fundamental group.

\end{proof}

\begin{remark}
\cite{DB} show that the generators $x_1,...,x_d$ in the presentation $$ <x_0,x_1,...,x_d| x_dx_{d-1}\cdots x_1 x_0^{-1}, x_0x_ix_0^{-1}x_i^{-1}, i=1,...,d>$$ are given by meridians around the $d$ lines in $\A$.
\end{remark}

\begin{remark}
Note that $W$ is a minimal space with 0-th, 1-st, and 2-nd integral homology being $\Z$, $\Z^d$, and $\Z^{d-1}$ respectively. So, the minimal CW-complex structure of $W$ is given by $S^1\times( \vee^{d-1} S^1)$.
\end{remark}

Part of the following theorem recalls lemma 5.1 in \cite{BM} with our convention described in this paper of twisted Alexander polynomials.

\begin{thm} Let $$\varepsilon: \pi_1(S^3\setminus K) \rightarrow \Z$$ be an epimorphism with $$\varepsilon(x_0) \neq 0$$ and  $$\rho:\pi_1(S^3\setminus K) \rightarrow GL(\mathbb{V}) = GL_r(\mathbb{F})$$ be a linear representation.

Then

\begin{itemize}
\item $\TA_i(S^3\setminus K,\Ft)$ are torsion over $\Ft$ for all $i\geq 0$.
\item $\TA_i(S^3\setminus K,\Ft)=0$ for $i\geq 2$. 
\item $\Delta_0^{\varepsilon,\rho}$ is the greatest common divisor of the $r\times r$ minors of the column matrix formed by $$\rho(x_i)t^{\varepsilon(x_i)} -Id,$$ where $i=0,...,d-1$.
\item $\Delta_1^{\varepsilon,\rho}/\Delta_0^{\varepsilon,\rho} = (\det(\rho(x_0)t^{\varepsilon(x_0)}-Id))^{d-2}$.
\end{itemize}

\end{thm}

\begin{proof}
We use Fox calculus (\cite{FC}, \cite{KL}) and theorem 4.1 in \cite{KL}. Since $\varepsilon$ is onto, $\TA_0(S^3\setminus K,\Ft)$ is torsion over $\Ft$ (cf. Proposition 2.1). By the cellular structure of $W$, we obtain the twisted chain complex as follows:
$$0\rightarrow \Ft^{r(d-1)} \xrightarrow{\partial_2} \Ft^{rd} \xrightarrow{\partial_1} \Ft^r\rightarrow 0.$$ This proves that any $i$-th twisted Alexander module with $i\geq 3$ is trivial. Following \cite{KL}, $\partial_1$ is the column matrix with $i$-th entry $$\rho(x_i)t^{\varepsilon(x_i)} - Id,$$ which yields the desired description of $\Delta_0^{\varepsilon,\rho}$.

By applying Fox calculus to the presentation $$\pi_1(S^3\setminus K) \cong <x_0,x_1,...,x_{d-1}| x_0x_ix_0^{-1}x_i^{-1}, i=1,...,d-1>,$$ we know $\partial_2$ is a $(d-1)\times d$ matrix with entries in $M_r(\Ft)$ given by the matrix of Fox derivatives of the relations, tensored with $\Ft^r$. Therefore, $\partial_2$ equals
{\footnotesize{\[ \left( \begin{array}{ccccc}
Id-\rho(x_1)t^{\varepsilon(x_1)} & \rho(x_0)t^{\varepsilon(x_0)}-Id & 0 & \cdots & 0 \\
Id-\rho(x_2)t^{\varepsilon(x_2)} & 0 & \rho(x_01)t^{\varepsilon(x_0)}-Id & \cdots & 0 \\
\vdots & \vdots& \vdots & \vdots & \vdots \\
Id-\rho(x_{d-2})t^{\varepsilon(x_{d-2})} & 0 & \cdots & \rho(x_0)t^{\varepsilon(x_0)}-Id & 0 \\
Id-\rho(x_{d-1})t^{\varepsilon(x_{d-1})} & 0 & \cdots & 0 & \rho(x_0)t^{\varepsilon(x_0)}-Id \end{array} \right)\]

}}

$\varepsilon(x_0)\neq 0$ guarantees that Ker($\partial_2$)=0. Therefore, $\TA_2(S^3\setminus K,\Ft)=0.$ Finally, using the fact that $\chi(S^3\setminus K)=0$ (recall $b_0=1, b_1=d, b_2=d-1$), we obtain $$\text{rank}_{\Ft}(\TA_1(S^3\setminus K,\Ft)) = -\chi(S^3\setminus K) =0.$$ Hence the first twisted Alexander module is also torsion. Then by theorem 4.1 in \cite{KL}, we obtain that $$\Delta_1^{\varepsilon,\rho}/\Delta_0^{\varepsilon,\rho} = (\det(\rho(x_0)t^{\varepsilon(x_0)}-Id))^{d-2}.$$

\end{proof}

\subsection{Plane curve $A_{\text{odd}}$ Singularities}

Let $C=\{ x^2-y^{2n}=0\} \subset \C^2$. The germ $(C,0)$ is known as the $\textbf{A}_{2n-1}$ singularity. $C$ is the union of two smooth curves which intersect non-transversely. By \cite{AO}, $$\pi_1(\C^2\setminus C) \cong G(2,2n) = <a_i,\beta | \beta = a_1a_0, R_1, R_2>,$$ where $$R_1: a_{i+2n} =a_i, i=0,...,2n-1$$ and $$R_2: a_{i+2} = \beta^{-1}a_i\beta , i=0,...,2n-1.$$

Therefore, $\pi_1(\C^2\setminus C)\cong$  $\begin{array}{c}
<a_0,a_1,...,a_{2n-1},\beta | a_1a_0\beta^{-1},\\
\beta a_2 \beta^{-1}a_0^{-1}, \beta a_4 \beta^{-1}a_2^{-1},...,\beta a_0 \beta^{-1} a_{2n-2}^{-1},\\
\beta a_3 \beta^{-1}a_1^{-1}, \beta a_5 \beta^{-1}a_3^{-1},...,\beta a_1 \beta^{-1} a_{2n-1}^{-1}>

\end{array}$

The untwisted Alexander matrix for $\textbf{A}_3$ singularity, from which one can get the general pattern, is:

\[ \left( \begin{array}{ccccc}
-1 & 0 & \beta & 0 & 1-a_0 \\
\beta & 0 & -1 & 0 & 1-a_2 \\
0 & -1 & 0 & \beta & 1-a_1 \\
0 & \beta & 0 & -1 & 1-a_3 \\
a_1 & 1 & 0 & 0 & -1
\end{array} \right)\]

It is easy to compute that the associated twisted Alexander matrix for $\textbf{A}_{2n-1}$ has trivial kernel. Hence, $\partial_2$ in the twisted chain complex has trivial kernel. Therefore, $\TA_1(\C^2\setminus C,\Ft)$ is a torsion module.

\begin{cor}
Let $C$ be the curve given by $y(y-x^n)=0 \subset \C^2$. Then the first twisted Alexander module of $\C^2\setminus C$ is a torsion $\Ft$-module.
\end{cor}
\begin{proof}
Use the fact that the germs $(y(y-x^n),0)$ and $(y^2-x^{2n},0)$ are diffeomophic. 
\end{proof}

\section{Twisted Alexander polynomials of complex curves}
\subsection{Definition}

The understanding of $\varepsilon:\pi_1\rightarrow \Z$ is desired to define the twisted Alexander modules. Since $\Z$ is an abelian group, $\varepsilon$ factors through the abelianization map from $\pi_1$ to $H_1$. Because of the nice geometric description for the first homology of complex curves complements, characterizing all associated $\varepsilon$'s is possible.

Let $C$ be a reduced curve in $\C \mathbb{P}^2$ of degree $d$ with $r$ irreducible components and let $L$ be a line in $\C \mathbb{P}^2$. Denote $$U=\C \mathbb{P}^2 \setminus (C\cup L) = \C^2 \setminus (C\setminus (C\cap L))$$ using the natural identification of $\C^2$ with $\C \mathbb{P}^2 \setminus L$. Alternatively, let $f(x,y): \C^2 \rightarrow \C$ be a square-free polynomial of degree $d$. Then projectivize $f=0$ and call $C$ the zero locus in $\C \mathbb{P}^2$. $L$ is given by $z=0$ and $U=\C^2 \setminus \{f=0\}$.

It is well known that $H_1(U) \cong \Z^r$ is generated by the homology classes $\nu_i$ of meridians $\gamma_i$ of the components $C_i$ of $C$, see corollary 4.1.4 in \cite{DB}.

Let $n_1,...,n_r$ be positive integers with gcd($n_1,...n_r$)=1. Then $\varepsilon: \pi_1 (U) \xrightarrow{ab} H_1(U) \xrightarrow{\psi: \nu_i \rightarrow n_i} \Z$ defines an epimorphism. If all $n_i =1$, then $\varepsilon$ is the total linking number homomorphism $$lk\#: \pi_1 (U) \xrightarrow{[\alpha] \rightarrow lk(\alpha,C \cup -dL)} \Z.$$ The consideration of some $n_i > 1$ allows us to study non-reduced polynomials $f$ with the algebraic linking number homomorphism.

Fix a finite dimensional $\mathbb{F}$-vector space $\mathbb{V}$ and a linear representation $\rho: \pi_1(U) \rightarrow GL(\mathbb{V}).$ Using section 2, the $\Ft$-modules $\TA_i(U;\Ft)$ are defined for all $i$ and are called the $i$-th twisted Alexander modules of $C$ with respect to $L$. 

If $L$ and $L'$ are two generic lines to $C$ in $\C\mathbb{P}^2$, then $U_L$ is homotopy equivalent to $U_{L'}$. So, if the choice of $L$ is made generically, then the notion of affine curve complement $U$ of $C$ is well defined up to homotopy. 

When we focus on the twisted Alexander modules associated to the total linking number homomorphism, we will denote $$H_i^{\rho}(U,\Ft) = H^{lk\#,\rho}_i(U,\Ft)$$ to indicate that the choice $\varepsilon = lk\#$ is understood.

\subsection{Motivation}

Above definition generalizes the classical Alexander modules. Let $\tilde{U}$ be the universal covering of $U$. Take $\mathbb{V}=\mathbb{F}=\Q$, $n_i=1$, and $\rho$ to be trivial. Then 
\begin{align*}
(\Ft \otimes \mathbb{V}) \otimes_{\mathbb{F}[\pi]}C_*(\tilde{U};\mathbb{F})& =(\Qt\otimes \Q) \otimes_{\mathbb{\Q}[\pi]}C_*(\tilde{U};\Q) \\
                & \cong \Qt\otimes_{\mathbb{\Q}[\pi]}C_*(\tilde{U};\Q)\\
                & = C_*(U;\Qt).
\end{align*}

The last equality can be found in \cite{AH}. The homology of the complex $C_*(U;\Qt)$ yields the Alexander modules studied in \cite{AC}.

Because $U$ is the complement of an affine curve, $U$ is isomorphic to a smooth affine hypersurface in $\C^3$. Hence $U$ is a complex 2-dimensional Stein manifold. Therefore $U$ has the homotopy type of a real 2-dimensional finite CW-complex (see 1.6.8 in \cite{DB} for details). Hence, $\TA_i(U;\Ft) =0$ for $i>3$ and $\TA_2(U;\Ft)$ is a free $\Ft$-module. For $i=0,1$, $\TA_i(U;\Ft)$ are finite type over $\Ft$.

Interesting questions are: What are the free ranks of $\TA_i(U;\Ft)$ over $\Ft$? In particular, under what conditions are these modules torsion over $\Ft$? As an example and partial answer to above questions, we recall some classical results.

\begin{thm}\cite{AC} If $X$ is a connected finite CW-complex with $H_1(X,\Z)\cong \Z$, then $H_1(X,\Qt)$ is a torsion $\Qt$-module.
\end{thm}

\begin{thm} \cite{DP} Suppose $H\subset \CP$ is a generic hyperplane to the hypersurface $V\subset \CP$. Then $V\setminus H$ has the homotopy type of a bouquet of spheres of dimension $n$.
\end{thm}

\begin{cor} Suppose $C\subset \C\mathbb{P}^2$ is an irreducible curve with line at infinity $L$, then $H_1(U,\Qt)$ is torsion over $\Qt$. Moreover, if $L$ is generic to $C$ and $C\setminus L$ is homotopic to a bouquet of $p$ circles, then $rank_{\Qt}H_2(U;\Qt)=p$.
\end{cor} 

\begin{proof}
It is easy to see that $\chi(U)= p$, by the additivity of Euler characteristics. On the other hand, $H_1(U)\cong \Z$ implies $H_1(U,\Qt)$ is torsion (cf. theorem 3.1). Also, $H_0(U,\Qt)\cong \Qt/(t-1) \cong \Q$ is torsion over $\Qt$. In addition, we have the equation $$rank_{\Qt}H_0(U;\Qt)-rank_{\Qt}H_1(U;\Qt)+rank_{\Qt}H_2(U;\Qt) = \chi(U),$$ hence $rank_{\Qt} H_2(U,\Qt) = \chi(U) = p $.
\end{proof}

In the irreducible case, the above questions concerning the ranks and torsion property have been answered partially, and particularly, under the assumption that $L$ is generic.

The main goal of this part of the paper is to study the terms in the equation stated in the main theorem in \cite{TC}. We will first find conditions such that the first twisted Alexander modules of $U$ are torsion and provide obstructions on the twisted Alexander polynomials. We will see that under certain geometric conditions concerning the singularities, we have torsion local modules. Hence, the acyclicity assumption in the main theorem in \cite{TC} could be discarded in many cases.

\subsection{When are $\TA_i(U,\Ft)$ torsion?}

\begin{thm}
Let $C$ be a reduced complex projective curve. If $C$ is irreducible and $\rho$ is abelian, or, if $C$ is transversal to the line at infinity $L$, then the twisted Alexander modules $\TA_i(\mathbb{P}^2\setminus C\cup L,\Ft)$ are torsion, for $i=0,1$.
\end{thm}

\begin{proof}
If $C$ is irreducible and $\rho$ is abelian, \cite{NV} shows that the classical Alexander modules of an irreducible curve determine the twisted ones, and the proof in this case is identical to the proof of the classical case shown in \cite{AC} using the Milnor long exact sequence.

Now assume $L$ is transversal to $C$. Let $C$ have degree $d$. Then by \cite{AC}, the link at infinity is given by the equation $x^d=y^d$, which is exactly the generalized Hopf link described before.

By lemma 5.2 in \cite{AC}, we have that $$\pi_1(S^3\setminus K)\cong <x_0,x_1,...,x_d| x_dx_{d-1}\cdots x_1 x_0^{-1}, x_0x_ix_0^{-1}x_i^{-1}, i=1,...,d> \rightarrow \pi_1(U)$$ is surjective. Moreover, from section 7 in \cite{AC}, $\pi_1(U)$ and $\pi_1(S^3\setminus K)$ have the same generators. Relations in $\pi_1(U)$ are relations in $\pi_1(S^3\setminus K)$ in addition to those being described by monodromy about exceptional lines using Zariski-Van Kampen method. Therefore, $\varepsilon \circ i_*= \varepsilon$ and $\rho\circ i_* = \rho$ on generators, where $i_*$ are the induced maps of inclusion on the fundamental groups.

Up to homotopy type, $U$ is 2-dimensional, so $U$ has the homotopy type of $S^3\setminus K$ with some 2-cells attached. Hence $$\TA_i(S^3\setminus K,\Ft) \rightarrow \TA_i(U,\Ft)$$ is an isomorphism for $i=0$, and an epimorphism if $i=1$. As a result, by theorem 2.8, it is sufficient to show that $\varepsilon\circ  i_* (x_0) =\varepsilon(x_0)\neq 0$. Notice that we have a commutative diagram:

$$\xymatrix{
\pi_1(S^3\setminus K)\ar[d]^{ab} \ar[r]^{i_*} & \pi_1(U)\ar[r]^{\varepsilon} \ar[d]^{ab} & \mathbb{Z}
\\
H_1(S^3\setminus K) \ar[r]^{j_*} & H_1(U)   \ar[ru]^{\psi}  }$$
 
So, $\varepsilon\circ i_* = \psi\circ j_*\circ ab$ and it is enough to understand the maps $ab$ and $j_*$.
 
Recall that $W$ is the complement of a central line arrangement $\A$ of $d$ lines in $\C^2$. $H_1(W) = \Z^d = <\nu_1,...,\nu_d>$, where $\nu_i$ is the homology class of the meridians around the line $l_i \subset \A$. Hence, by remark 2.6, we have $ab(x_i) = \nu_i$ for $i=1,...,d$. Therefore, $ab(x_0) = \nu_1+...+\nu_d$.

On the other hand, $H_1(U) = \Z^r$, generated by the homology classes $\beta_l$ of the meridians around each irreducible component. Since $\A$ is defined by the homogeneous part of the defining equation of $C$, it is clear that $j_*$ takes each $\nu_i$ to one of the $\beta_l$'s. We know that there are exactly $d_l$ many $\nu_i$ mapping to $\beta_l$, where $d_l$ is the degree of the component $C_l$ of $C$. 

Finally, since $\psi(\beta_l) =n_l$, then for all $i\geq 1$, $\varepsilon\circ i_*(x_i) =n_{l_i}$ for some $l_i$, and $$\varepsilon\circ i_*(x_0) = \psi\circ j_* (\nu_1+...+\nu_d) =\sum_{l=1}^r d_ln_l> 0.$$

\end{proof}

\begin{remark}
For the space $W = \C^2\setminus\A $, where $\A$ is a central line arrangement of $d$ lines, $H_1^{\rho}(W,\Ft)$ is torsion over $\Ft$ also because $$H_1^{\rho}(W,\Ft) \cong H_1(W_{\infty},\mathbb{V}_{\rho}).$$ $W_{\infty}$ is homotopy equivalent to the Milnor fiber at the origin of the polynomial $x^d-y^d=0$. $W_{\infty}$ has the homotopy type of a finite bouquet of circles (\cite{SP}). Therefore, $H_2^{\rho}(W,\Ft) =0$ and $H_1^{\rho}(W,\Ft)$ is a finite $\mathbb{F}$-vector space, and hence a torsion $\Ft$-module. Similar arguments will be given later to show that some other Alexander modules are torsion.
\end{remark}

Even though the Milnor fiber argument is a short proof for torsionness, it does not give following obstructions for $\Delta_{1,U}^{\varepsilon,\rho}$, which follows from theorem 2.8 and theorem 3.4.

\begin{cor}
If $C$ is in general position at infinity, then $\Delta_{1,U}^{\varepsilon,\rho}$ divides $$\gcd(\det(\rho(x_0)t^{\sum_{l=1}^r d_ln_l}-Id), \det(\rho(x_1)t^{n_{l_1}}-Id),...,\det(\rho(x_{d-1})t^{n_{l_{d-1}}}-Id))\cdot(\det(\rho(x_0)t^{\sum_{l=1}^r d_ln_l}-Id))^{d-2}.$$ 

In particular, if $\varepsilon = lk\#$, then $\Delta_{1,U}^{\rho}$ divides $$\gcd(\det(\rho(x_0)t-Id),...,\det(\rho(x_{d-1})t-Id))\cdot(\det(\rho(x_0)t^d-Id))^{d-2}.$$
\end{cor}

\begin{remark}
This corollary generalizes the formula in the classical case. Namely, when $\varepsilon=lk\#$ and $\rho$ is trivial, then we have the formula $(t-1)(t^d-1)^{d-2}$ in \cite{AC} for the Alexander polynomial at infinity.
\end{remark}

\begin{remark}
Corollary 3.6 is very similar to theorem 3.1 in \cite{TS}. But their definition of twisted Alexander polynomials follows \cite{WD}. The two definitions are related in theorem 4.1 in \cite{KL}, provided that the first twisted Alexander module of a space $X$ is torsion. 

\end{remark}

\begin{remark}
In \cite{NV}, Libgober proved that for an irreducible curve $C$, and for $\rho$ a unitary representation, the roots of the first twisted Alexander polynomial are in a cyclotomic extension of the field generated by the rationals and the eigenvalues of $\rho(\gamma)$, where $\gamma$ is a meridian around a non-singular point in $C$. But his result does not give the extension degree. By corollary 3.6, we actually know what the field extension is and therefore the degree in some cases.
\end{remark}

\begin{cor}
Suppose $\mathbb{F} = \C$. Denote the eigenvalues of $\rho(x_0)^{-1}$ by $\lambda_1,...,\lambda_r$. Then the roots of $\Delta_{1,U}^{\rho}$ lie in the splitting field $S$ of $\prod_{i=1}^r (t^d-\lambda_i)$ over $\Q$, which is cyclotomic over $K=\Q(\lambda_1,...,\lambda_r)$.
\end{cor}

\begin{proof}
If $\rho(x_1),...,\rho(x_d)$ have no common eigenvalues, then $\Delta_{1,U}^{\rho}$ divides $(\det(\rho(x_0)t^d-Id))^{d-2}$. In particular, prime factors of $\Delta_{1,U}^{\rho}$ are among prime factors of $\det(\rho(x_0)t^d-Id)$.

Let $p(t)$ be the characteristic polynomial of $\rho(x_0)^{-1}$.
Then it is clear that $$\det(\rho(x_0)t^d-Id) = (-1)^r\det(\rho(x_0))p(t^d)=(-1)^r\det(\rho(x_0))(t^d-\lambda_1)\cdots (t^d-\lambda_r).$$
Then the roots of $\Delta_{1,U}^{\rho}$ are inside the splitting field $S$ of $\prod_{i=1}^r (t^d-\lambda_i)$ over $\Q$.

If $\alpha$ is a common eigenvalue of $\rho(x_1),...,\rho(x_d)$, then one of the eigenvalues of $\rho(x_0)=\rho(x_d)\rho(x_{d-1})...\rho(x_1)$ is $\alpha^d$. And hence WLOG, $\alpha^d=\lambda_1^{-1}$. So $\alpha\in S$.
\end{proof}

\begin{example}
In some cases, the degree $[S:K]$ is easy to compute.
\begin{enumerate}
\item If all eigenvalues of $\rho$ are transcendental and they have no algebraic relations, then the degree of extension is $d^r\varphi(d)$, where $\varphi$ is the Euler function.
\item If $\rho$ is unitary and all eigenvalues of $\rho$ are algebraic (e.g. $\rho$ only takes values in matrices with algebraic elements), and write $\lambda_i=e^{2\pi i/k_i}$, then $$[S:K] = \frac{\varphi(lcm(d,k_1,...,k_r))}{\varphi(lcm(k_1,...,k_r))}.$$
\item If $\rho$ is unitary and there are $m$ transcendental elements in which there are no relations, by writing the algebraic $\lambda_i=e^{2\pi i/k_i}$, we have $$[S:K] = \frac{d^m\varphi(lcm(d,k_i's))}{\varphi(lcm(k_i's))}.$$
\end{enumerate}

\end{example}

\subsection{Local twisted Alexander modules}

In this section we consider the local Alexander modules of singular points of a complex curve. These are defined to be the twisted Alexander modules of the link exterior at each singular point (see \cite{TC}). There are two types of singular points according to their positions; interior singular points and singularities at infinity (intersections of $C$ and $L$). The main result of this section will be the implication of torsionness of the global twisted Alexander modules from those of the local ones at infinity.

If $x$ is an interior singular point of a curve $C$, consider a small sphere $S_x^3$ around $x$. $L_x:=S_x^3\cap C$ is called the algebraic link at $x$. Diffeomorphism type of $L_x$ does not change when $S_x^3$ is small enough (\cite{DB},\cite{SP}). Twisted Alexander modules at $x$ are defined as $\TA_i(S_x^3\setminus L_x,\Ft)$ via the induced map from inclusion on $\pi_1$. 

For all points x in C, there is a locally trivial fiber bundle, called the Milnor fibration (\cite{SP}), $$F_x \rightarrow S_x^3\setminus L_x \xrightarrow{h/|h|} S^1,$$ where $h$ is the defining equation of the germ $(C,x)$ and $F_x$ is the Milnor fiber at $x$. $F_x$ has the homotopy type of a finite bouquet of circles and the number of circles is the Milnor number of $f$ at $x$.

If $x$ is a singular point at infinity, then we consider a small sphere $S_x^3$ centered at $x$ in $\mathbb{CP}^2$ and $L_x = S_x^3\setminus (C\cup H)$. The twisted Alexander modules at $x$ are again defined through the inclusion maps on the fundamental group.

\begin{prop}
Local twisted Alexander modules at interior singular points are torsion $\Ft$-modules.
\end{prop}

\begin{proof}
Suppose $x$ is a singular point located in exact one component $C_i$. Let $E=S_x^3\setminus L_x$. The composition $\varepsilon\circ i_*:\pi_1(E) \rightarrow \Z$ takes generators to $n_i$. 

We proceed by considering two `infinite cyclic covers' of $E$.  Let $E_{n_i,\infty}$ denote the infinite cyclic cover of $E$ associated to $\varepsilon\circ i_*$.  On the other hand, the map $\pi_1(E) \rightarrow \pi_1(U) \xrightarrow{\nu_i \rightarrow 1} \Z$ defines $E_{\infty}$ another cover of $E$. By \cite{KL},$E_{n_i,\infty}$ is the disjoint union of $n_i$ copies of $E_{\infty}$. Moreover, it is well known that $E_{\infty}$ is homotopic equivalent to $F_x$.

Therefore, $\TA_i(E,\Ft) \cong H_i(E_{n_i,\infty},\mathbb{V}_{\rho}) \cong H_i(\cup^{n_i}F_x,\mathbb{V}_{\rho})$. The latter are finite vector spaces, hence torsion $\Ft$ modules. Moreover, $\TA_2(E,\Ft)=0$.

If $x$ is a singular point which lies on more than one component, say $C_1$ and $C_2$, the image of $\varepsilon\circ i_*$ is $(n_1,n_2)\Z$. A similar argument as above shows that $\TA_i(E,\Ft) \cong  H_i(\cup^{(n_1,n_2)}F_x,\mathbb{V}_{\rho})$. 

\end{proof}

\begin{remark}
The relation between the twisted Alexander polynomials of $H_1(E_{\infty},\mathbb{V}_{\rho})$ and $H_1(E_{n_i,\infty},\mathbb{V}_{\rho})$ is stated in \cite{KL}. One only needs to replace $t$ by $t^{n_i}$ in the polynomial.
\end{remark}

\begin{thm}
If there exists a component of $C$ such that its local twisted Alexander modules at infinity are torsion, then the first (global) twisted Alexander module of the curve $C$ is torsion.
\end{thm}

\begin{proof}
Let $C_1$ be such component of $C$. Let $T(C_1)$ be a regular neighborhood of $C_1$ in $\mathbb{CP}^2$ and $T=T(C_1)\setminus (C\cup H)$. By Lefschetz hyperplane theorem, there is a surjection $\pi_1(T)\rightarrow \pi_1(U)$. Hence, there is a surjective $\Ft$-module homomorphism $\TA_1(T,\Ft) \rightarrow \TA_1(U,\Ft)$. It suffices to show that $\TA_1(T,\Ft)$ is a torsion $\Ft$-module.

Let $F_1$ be the surface obtained from $C_1$ with a disk around each singular point of $C_1$, including at infinity, removed (cf theorem 5.6 in \cite{TC}). Let $N=F_1\times S^1$. Its boundary will be a union of disjoint tori $T^2$'s. Therefore, $T\simeq \partial T(C_1)$ is homotopic to the decomposition $$N \cup\bigcup_{x\in \text{Sing}(C_1)\cup \text{Sing}_{\infty}(C_1)} S_x\setminus L_x$$

There is an associated Mayer-Vietoris sequence $$...\rightarrow \oplus\TA_1(T^2,\Ft) \rightarrow \TA_1(N,\Ft) \oplus \TA_1(S_x\setminus L_x,\Ft) \rightarrow \TA_1(T,\Ft) \rightarrow \oplus\TA_0(T^2,\Ft)\rightarrow ...$$

Twisted Alexander modules of tori are torsion (\cite{KL}). $\pi_1(N)=\pi_1(F_1\times S^1)\cong \Z\times \Z^{*b_1(F_1)}$, free group of rank $b_1(F_1)$. Since $\varepsilon$ takes the meridian around $C_1$ to $n_1>0$, by the calculations in section 2, twisted Alexander modules of $N$ are also torsion. 

\end{proof}

\begin{cor} 
If the singularities at infinity of a single component of $C$ are of type $x^k-y^k$ or  $\textbf{A}_{2n-1}$, then the first twisted Alexander module of $C$ is torsion. 
\end{cor}

\begin{cor} 
Let $C=C_1\cup...\cup C_r$ be an affine plane curve which is transversal at infinity. Set $M=\C^2\setminus C$. Consider an epimorphism $\varepsilon: \pi_1(M) \rightarrow \Z$ and an unitary representation $\rho:\pi_1(M) \rightarrow GL(\mathbb{V})$. Let $s=\#Sing(C)$ and consider $(S^3_k,L_k)$ the local link singularities, $k=1,2,...,s$, plus the link at infinity $(S^3_{\infty},L_{\infty})$. Then 
$$\alpha\cdot \prod_{k=1,...,s,\infty} \Delta^{\varepsilon,\rho}(S^3_k\setminus L_k) =  \Delta^{\varepsilon,\rho}(M) \cdot\bar{\Delta^{\varepsilon,\rho}(M)}\cdot det(\phi^{\varepsilon,\rho}(M)), $$
where $\alpha = \prod_{q=1}^r det(Id-\rho(\nu_q)t^{\varepsilon(\nu_q)})^{s_q-\chi(C_q)}$, with $s=\#Sing(C)\cap C_q$ and $\nu_q$ is a meridian of $C_q$, and $\phi^{\varepsilon,\rho}(M)$ is the intersection form described in \cite{CF}.
\end{cor}

\subsection{Examples and conjectures}

\begin{example}
$C: x^2-y^2 \subset \C^2$.

From \cite{KL}, we know $\Delta=\Delta_1/\Delta_0 = 1$ and hence $$\Delta_1=\Delta_0 = \gcd(\det(\rho(x_1)t^{n_1} -I), \det(\rho(x_2)t^{n_2}-I)).$$

Notice that the global Reidemeister torsion (\cite{TC} for definition) and the local Reidemeister torsion at the singular point are the same because $x^2-y^2$ is a central line arrangement. Therefore, by theorem 5.6 in \cite{TC}, the determinant of the intersection form $\det\varphi^{\varepsilon,\rho}(C)$ is equal to $$\det((I-\rho(x_1)t^{n_1})(I-\rho(x_2)t^{n_2})).$$

\end{example}

\begin{example}
$C$ is the generic line arrangement of 3 lines in $\C^2$. 

We know that $\TA_2(U,\Ft)=0$ and $\TA_1(U,\Ft)$ is torsion. Also all the local twisted Alexander modules are torsion. Theorem 4.1 in \cite{KL} shows

$$\Delta=\Delta_1/\Delta_0 = \gcd(\det(\rho(x_1)t^{n_1} -I), \det(\rho(x_2)t^{n_2}-I),\det(\rho(x_3)t^{n_3}-I)).$$ 

All the local Reidemeister torsion at the three singular points are 1 because the link exteriors are tori. The Reidemeister torsion at infinity is computed using the result on Hopf link with 3 components, $$\Delta_{\infty} = \det(\rho(x_3x_2x_1)t^{n_1+n_2+n_3}-I).$$ 

\end{example}

\begin{example}
$C:(x^2-y^3)(x-1) \subset \C^2$. 

Since the two components intersect transversally, by \cite{OKA}, $$\pi_1(U) \cong <x,y,z|xyx=yxy, xz=zx, yz=zy>.$$

There are three $\textbf{A}_1$ singularities and one $\textbf{A}_2$ singularity. We know that $\TA_2(U,\Ft)=0$ and $\TA_1(U,\Ft)$ is torsion. Also all the local twisted Alexander modules are torsion.

$$\Delta=\Delta_1/\Delta_0 = \gcd(\det(\rho(x)t^{n_x} - \rho(y)\rho(x)t^{n_x+n_y} -I), \det(\rho(z)t^{n_z}-I))).$$ 

Local Reidemeister torsion at the three $\textbf{A}_1$'s are 1. The Reidemeister torsion at infinity is computed using the result on Hopf link with 4 components. That at $\text{A}_2$ is given by $$\Delta_{A_2} = \frac{\det(\rho(x)t^{n_x} - \rho(y)\rho(x)t^{n_x+n_y} -I)}{\det(I-\rho(x)t^{n_x})}.$$

\end{example}

Recall that the classical Alexander polynomial of curves has a weakness. Namely, if a curve $C$ is the union of irreducible curves which intersect transversely, then its Alexander polynomial is trivial, being $(t-1)^{r-1}$(\cite{OKA}). In the previous example, we saw that the twisted Alexander polynomial is ``non-trivial" in some sense, even though we have the two components intersecting transversely. More precisely, we saw $\Delta \neq 1$ for some representation $\rho$. But $\Delta$ is trivial in the first example when the two components are both smooth.

\begin{con}
Suppose all components of $C$ intersect transversely. Then for all $\rho$, $\Delta_1/\Delta_0 =1$ if and only if $C$ is union of two irreducible nodal curves.
\end{con}

Here is a verifying example:
\begin{example}
Suppose $C$ is the union of $C_1$ and $C_2$, where $C_1,C_2$ are irreducible and they intersect transversely. Assume $C_1$ is diffeomorphic to $x^{p_1}-y^{q_1} \subset \C^2$, where $(p_1,q_1)=1$ and $C_2$ is diffeomorphic to $x^{p_2}-y^{q_2} \subset \C^2$, where $(p_2,q_2)=1$. 
By \cite{AO} and \cite{OKA}, we know
$$\pi_1(U) \cong <x,y,z,w| x^{p_1}y^{-q_1}, z^{p_2}w^{-q_2},xzx^{-1}z^{-1},xwx^{-1}w^{-1},yzy^{-1}z^{-1},ywy^{-1}w^{-1}>.$$

By fixing a generic line to $C$, we have that $\TA_1(U,\Ft)$ is torsion over $\Ft$.
It is also clear that all local twisted Alexander modules are torsion. Take $\varepsilon=lk\#$ and $\rho$ to be rank-1. Take $\rho(x)$ to be a non-trivial $p_1$-th root of unity and $\rho(y)$ to be a non-trivial $q_1$-th root of unity. Take $\rho(z)=1=\rho(w)$. Then by theorem 4.1 in \cite{KL}, $\Delta=\Delta_1/\Delta_0 = t-1 \neq 1$ and $\Delta_0 =1$. Similar results can be obtained using similar choice of $\rho$.
\end{example}

\section{Twisted Alexander polynomials of complex hypersurfaces}

We generalize the notion of twisted Alexander modules to the study high-dimensional complex hypersurface complements. A major difference is that for curves, there are only isolated singularities, while for high dimensional hypersurfaces, the singularities could be non-isolated. We will make the idea of ``local" modules explicit. 

\subsection{Definition}

Let $V$ be a reduced hypersurface in $\CP$ of degree $d$ with $r$ irreducible components and $H$ be a hyperplane in $\CP$. Let $$U=\CP \setminus (V\cup H) = \C^{n+1} \setminus (V\setminus (V\cap H))$$ with the natural identification of $\C^{n+1}$ and $\CP \setminus H$. Alternatively, let $f(x_1,...,x_{n+1}): \C^{n+1} \rightarrow \C$ be a reduced polynomial of degree $d$. Then $V$ is given by projectivizing $f=0$. $H$ is given by $x_0=0$.

Notice that $H_1(U) \cong \Z^r$ and is generated by the homology classes $\nu_i$ of meridians $\gamma_i$ of the components $V_i$ of $V$ (\cite{DB}). Let $n_i$ be $r$ positive integers with gcd($n_1,...n_r$)=1. Let $$\varepsilon: \pi_1 (U) \xrightarrow{ab} H_1(U) \xrightarrow{\nu_i \rightarrow n_i} \Z$$, which is an epimorphism. 

We obtain the linking number homomorphism if $n_i=1$, $$lk\#: \pi_1 (U) \xrightarrow{[\alpha] \rightarrow lk(\alpha,V \cup -dH)} \Z.$$

Fix a finite dimensional $\mathbb{F}$-vector space $\mathbb{V}$ and a linear representation $\rho: \pi_1(U) \rightarrow GL(\mathbb{V}).$

Using section 2, the $\Ft$-modules $\TA_i(U;\Ft)$ are defined for all $i$ and are called the $i$-th twisted Alexander modules of the hypersurface $V$. Again we denote $$H_i^{\rho}(U,\Ft) = H^{lk\#,\rho}_i(U,\Ft).$$

Because $U$ is the complement of an affine hypersurface, it is isomorphic to a smooth affine hypersurface in $\C^{n+2}$, by 1.6.7 in \cite{DB}. Hence $U$ is a Stein manifold of complex dimension $n+1$ and therefore $U$ has the homotopy type of a finite CW-complex of real dimension $n+1$ (\cite{DB}). As a result, $\TA_i(U;\Ft) =0$ for $i>n+1$ and $\TA_{n+1}(U;\Ft)$ is a free $\Ft$-module. Moreover for $0\leq i \leq n$, $\TA_i(U;\Ft)$ is of finite type over $\Ft$.

\subsection{When are $\TA_i(U,\Ft)$ torsion?}

We want to address the following questions: What are the free ranks over $\Ft$? When are these modules torsion over $\Ft$?

Classical theory tells us some partial answers:

\begin{thm}\cite{MT} Suppose $H$ is a generic hyperplane to $V$ ($H$ is transversal to $V$ in the stratified sense). Then the Alexander module $H_i(U;\Qt)$ is a torsion $\Qt$-module, and hence a finite dimensional vector space over $\Q$, for $0\leq i\leq n$. (also see \cite{RF})
\end{thm}

Combining with theorem 3.2 yields the following.

\begin{cor} Suppose $H$ is a generic hyperplane to $V$. Then $rank_{\Qt}H_{n+1}(U;\Qt)=\mu$, the number of $n$-spheres in the bouquet describing the homotopy type of $V\setminus H$.
\end{cor} 

The proof of the above corollary is identical to the case of curve complements. The result below follows from our calculation on curves.

\begin{cor}
Let $V$ be a reduced complex projective hypersurface which is transversal to the hyperplane at infinity $H$. Then the first twisted Alexander module $\TA_1(U,\Ft)$ is torsion.
\end{cor}

\begin{proof}
This is a corollary of theorem 3.4 and the Lefschetz hyperplane theorem. Let $E$ be a 2-dimensional plane in $\CP$ which is generic to $V\cup H$. Then there is an epimorphism $\pi_2(E\cap U) \rightarrow \pi_2(U)$ and an isomorphism $\pi_1(E\cap U) \rightarrow \pi_1(U)$. Hence, $\TA_i(E\cap U,\Ft)$ surjects on $\TA_i(U,\Ft)$ for i=0,1,2.
\end{proof}

In the above statements, the assumption that $H$ is generic plays an important role. Moreover, if $H$ is generic, there are some other very interesting properties about classical Alexander modules (\cite{MT}). In short, the Alexander polynomials would divide the products of some `locally defined polynomials' and their roots are among $d$-th roots of unity.

So, starting from this section, we will assume $H$ to be generic to $V$. In particular, $H$ intersects $V$ transversally (also in the stratified sense). 

From lemma 1.5 in \cite{HG}, if $\dim Sing V \leq \dim V -2$ (which implies $V$ is irreducible), then $\pi_1(U) \cong \Z$. In this case, $\rho$ is abelidan and is determined by the image of the meridian around the hypersurface. We see that the twisted polynomials are determined by the classical ones (\cite{NV}). Therefore, the twisted modules are only interesting with non-abelian representation and when the singularities are in codimension 2 (e.g. $V$ is a hyperplane arrangement).

Finally, we state a result about the free rank of $\TA_{n+1}(U;\Ft)$ over $\Ft$.

\begin{prop}
If $\TA_i(U;\Ft)$ are torsion over $\Ft$ for $0\leq i\leq n$, then the rank  of $\TA_{n+1}(U;\Ft)$ over $\Ft$ is equal to $(-1)^{n+1}\chi(U)\cdot \dim_{\mathbb{F}}(\mathbb{V})$.
\end{prop}

\begin{proof}
Let $\tilde{U}$ be the universal cover of $U$. Consider the complex $C_*(\tilde{U};\mathbb{F})$ of free $\Ft$-modules. $$\sum_{i=0}^{n+1}(-1)^i \text{rank}_{\Ft}(C_i(\tilde{U};\mathbb{F})) = \chi(U).$$ So, $$\sum_{i=0}^{n+1}(-1)^i \text{rank}_{\Ft}(\TC_i(U;\Ft)) = \chi(U)\cdot\dim_{\mathbb{F}}(\mathbb{V}).$$ 
\end{proof}

\subsection{Divisibility at infinity and Torsionness}

From now on, we will focus on the case where $\varepsilon=lk\#$. We will prove that under the assumption of transversality, most of the twisted Alexander modules are torsion. Let $T(H)$ be a small regular tubular neighborhood of $H$. Denote the space $T(H) \setminus V\cup H$ by $T_H$. Then by \cite{RF},

\begin{equation}\label{*}
\pi_i(T_H) \xrightarrow{i_*} \pi_i(U)
\end{equation}

is an isomorphism for $1\leq i \leq n-1$ and is an epimorphism for $i=n$. Hence, $\pi_i(U,T_H)=0$ for $i\leq n$. As a result, $U$ has the homotopy type of a CW-complex built from $T_H$ by attaching cells of dimension $n+1$.

We define the twisted Alexander modules of $T_H$ by composing $\varepsilon$ and $\rho$ with the inclusion map on fundamental groups. The composition with $\varepsilon$ is surjective by (\ref{*}). Let $(T_H)_{\infty}$ be the infinite cyclic cover of $T_H$ associated to this composition. Then $(T_H)_{\infty}$ is connected and can be thought as a subspace of $U_{\infty}$. Moreover, the module action is compatible with the inclusion map. Namely, the following diagram of $\Ft$-modules commutes:

$$\xymatrix{
\mathbb{V}\otimes_{\rho\circ i} C_*((T_H)_{\infty}) \ar[d]^t \ar[r]^i & \mathbb{V} \otimes_{\rho} C_*(U_{\infty})\ar[d]^t\\
\mathbb{V}\otimes_{\rho\circ i} C_*((T_H)_{\infty}) \ar[r]^i & \mathbb{V} \otimes_{\rho} C_*(U_{\infty})}$$

Let $v\otimes c\in \mathbb{V} \otimes C_*((T_H)_{\infty})$. Let a meridian around $H$, $\alpha\in \pi_1(T_H)$ gets mapped to $\beta\in \pi_1(U)$ by inclusion, then
\begin{align*}
t(i(v\otimes c))& =t(v \otimes i(c)) \\
                & =v \beta^{-1} \otimes \beta i(c)\\
                & = \rho(\beta^{-1})(v) \otimes \beta i(c) \\
                &= (\rho\circ i)(\alpha^{-1})(v) \otimes i(\alpha) i(c)  \\
                & =  \alpha v  \otimes i(\alpha c)\\
                & = i(\alpha v \otimes \alpha c) \\
                & = i(t(v\otimes c)).
\end{align*}

As a result, $$\TA_i(T_H,\Ft) \xrightarrow{i_*} \TA_i(U,\Ft)$$ is an isomorphism for $i\leq n-1$ and is an epimorphism for $i=n$.

\begin{cor}
If $H_i((T_H)_{\infty},\mathbb{V}_{\rho})$ is a torsion module over $\Ft$, then so is $\TA_i(U,\Ft)$.
\end{cor}

\begin{thm}
$H_i^{\rho}(U,\Ft)$ are torsion $\Ft$-modules, for $0\leq i\leq n$.
\end{thm}

\begin{proof}
By \cite{RF}, \cite{MT}, $T_H$ is homotopy equivalent to the complement of the cone over $V\cap H$ in $\C^{n+1}$. For instance, if $V$ is defined by $f(x_0,...,x_n)=0$ and $H$ is defined by $x_0=0$, then the cone is defined by $g:= f(0,x_1,...,x_n)=0$. Furthermore, the latter space is homotopy equivalent to $S^{2n+1} \setminus \{g=0\}$, because $g$ is homogeneous. On the other hand, with the existence of Milnor fibration: $$F_g \rightarrow S^{2n+1} \setminus \{g=0\} \rightarrow S^1,$$ where $F_g$ is the Milnor fiber of the function $g$, when $\varepsilon=lk\#$, $(T_H)_{\infty}\cong F_g$. Because $F_g$ is a finite $n$-dimensional CW-complex, $H_i(F_g,\mathbb{V}_{\rho})$ is a finite vector space and a torsion $\Ft$-module.
\end{proof}

\subsection{Relation to local systems on $U$ with stalk $\mathbb{V}$}

We obtain here some combinatorial upper bounds on the multiplicities of roots of the twisted Alexander polynomials for hyperplane arrangements, similar to the work in \cite{DN} section 4.

Let $a\in \C^*$ and consider the rank-one local system $\mathcal{L}_a$ on $U$ defined by $$\pi_1(U) \xrightarrow{\varepsilon} \Z \xrightarrow{1\mapsto a} \C^*.$$

Then we have the short exact sequence tensoring with $\mathbb{V}$, $$0\rightarrow \Ct\otimes\mathbb{V} \xrightarrow{(t-a)I} \Ct\otimes \mathbb{V} \rightarrow \mathcal{L}_a\otimes\mathbb{V} \rightarrow 0.$$

There is an induced short exact sequence of the complexes $$0\rightarrow \TC_*(U,\Ct) \rightarrow  \TC_*(U,\Ct) \rightarrow C_*(U,\mathcal{L}_a\otimes\mathbb{V}_{\rho}) \rightarrow 0$$ which yields the Milnor long exact sequence $$...\rightarrow H_i^{\rho}(U,\Ct) \xrightarrow{(t-a)^r} H_i^{\rho}(U,\Ct) \rightarrow H_i(U,\mathcal{L}_a\otimes\mathbb{V}_{\rho}) \rightarrow ...$$

Since we know $H_i(U,\Ct)$ is torsion for $i\leq n$, exactness of the sequence shows  $$\dim_{\C}H_i(U,\mathcal{L}_a\otimes\mathbb{V}_{\rho}) = \dim(Ker(t-a)^r:H_i \rightarrow H_i)  + \dim(Ker(t-a)^r:H_{i-1} \rightarrow H_{i-1}) $$

\begin{prop}
$\dim_{\C}H_i(U,\mathcal{L}_a\otimes\mathbb{V}_{\rho}) \geq N(a,i)+ N(a,i-1)$, where $N(a,q)$ denotes the power of $(t-a)$ in $\Delta_q^{\rho}$.
\end{prop}

In the case of hyperplane arrangements, since the cohomology groups of hyperplane arrangement complements are combinatorial, more precisely, depending on the intersection lattice, we can give combinatorial upper bounds for multipilities of roots of twisted Alexander polynomials.

\begin{cor}
Suppose $U$ is a hyperplane arrangement complement in $\C^{n+1}$. Denote by $N(a,q)$ the power of $(t-a)$ in $\Delta_q^{\rho}$. Then for $i\leq n$, $$N(a,i)+N(a,i-1) \leq rb_i(U).$$
\end{cor}

This is due to the existence of a minimal cell structure on $U$. In general, $$\dim_{\C}H_i(U,\mathcal{L}_a\otimes\mathbb{V}_{\rho}) \leq r\cdot c_i(U),$$ where $c_i(U)$ is the number of cells of $U$ in the minimal cell structure.

\section{Divisibility theorem of Twisted Alexander polynomials}

\subsection{Topology of tubular neighborhood of components of $V$}

Let $V_1$ be a component of $V$. Let $T(V_1)$ be a tubular neighborhood of $V_1$ in $\CP$ (for an explicit construction of such tubular neighborhood, one can see \cite{DB} page 150). By a Bertini's theorem type argument, $T(V_1)$ contains a generic hypersurface $W$ of dimension $n$ and the same degree as $V_1$, which is transversal to $V\cup H$ (see proof of theorem 4.3 in \cite{HG} for the irreducible case). So, by Lefschetz theorem, the map on homotopy groups induced by inclusion:
$$\pi_i(W \cap (\CP\setminus (V\cup H))) \rightarrow \pi_i (\CP \setminus (V\cup H))$$ is an isomorphism for $i \leq n-1$ and is an epimorphism for $i=n$.

Let $T=T(V_1) \setminus (V\cup H)$. Then the above maps factor through $\pi_i(T)$ as:
$$\pi_i(W \cap (\CP\setminus (V\cup H))) \rightarrow \pi_i(T) \rightarrow \pi_i (U)$$ 
So,
\begin{equation}\label{**}
\pi_i(T) \xrightarrow{i_*} \pi_i(U)
\end{equation}

is an isomorphism for $i \leq n-1$ and is an epimorphism for $i=n$.

Following the argument in \cite{MT},\cite{L2}, we have $\pi_i(U,T) =0$ for $i\leq n$. Hence, $U$ has the homotopy type of $T$ with some cells of dimension $n+1$ or higher attached. So the low dimensional skeleton of $U$ and $T$ are homotopic.

Similar to the tubular neighborhood at infinity, we consider the twisted Alexander modules of $T$ corresponding to the representation of $\pi_1(T)$ obtained by composing $\varepsilon$ and $\rho$ with the inclusion map on fundamental groups. The composition with $\varepsilon$ is still surjective because of (\ref{**}). Let $T_{\infty}$ be the infinite cyclic cover of $T$ associated to this composition. Then $T_{\infty}$ is connected and can be thought as a subspace of $U_{\infty}$. Moreover, the module action is compatible with the inclusion map. 

Recall that $U$ has the homotopy type of $T$ with some cells of dimension $n+1$ or higher attached. So, the map induced by inclusions $$\TA_i(T;\Ft) \rightarrow \TA_i(U;\Ft)$$ is an isomorphism for $i\leq n-1$ and is an epimorphism for $i=n$.

To study the twisted Alexander modules of $V$, we may focus on the tubular neighborhood of only one of the components. 

\begin{cor}
Let  $V_{k_1}$, $V_{k_2}$,...,$V_{k_j}$ be components of $V$ such that the twisted Alexander modules of their tubular neigborhood are torsion. Then $\Delta_i^{\varepsilon,\rho}(t)$ divides the greatest common divisor of the twisted Alexander of polynomials of $V_{k_1}$, $V_{k_2}$,...,$V_{k_j}$.
\end{cor}

\subsection{Stratification of the hypersurface, partition of a tubular neighborhood, and local link pairs}

An advantage of studying a tubular neighborhood is the possibility writing itself as an union of some topoloically locally trivial spaces. 

Fixing a Whitney regular stratification of $V$, one can give a partition of $T$ consisting of total spaces of some locally trivial fibration, where the base spaces are strata in $V_1$. Libgober has described the stratification in detail for the special case of hyperplane arrangements in \cite{EM}.

Fix a Whitney regular stratification of $V$. The tubular neighborhood $T(V_1)\setminus V$ in $\CP$ can be decomposed into subsets; each of which corresponds to one strata in $V_1$, in such a way that it fibers over the stratum. Denote $\mathbf{X}_j^k$ as the $j$-th $k$ dimensional stratum in $V_1$ and $\mathbf{Y}_j^k$ as the associated subspace of the tubular neighborhood. For example, if $V$ has only isolated singularities, $\mathbf{X}^n$ will be the smooth part of $V$ and $\mathbf{X}_j^0$ can be regarded as the $j$-th singular point. Since the stratification of $V$ makes $V$ into a pseudomanifold, it is known that the fiber is given by the complement of $V$ in a small disk $\mathbb{D}_j^{2n-2k+2}$ in $\CP$ with codimension $k$ which is transversal to $\mathbf{X}_j^k$. Moreover, $\mathbb{D}_j^{2n-2k+2} \cap V$ is  homeomorphic to a cone of a $2n-2k-1$ dimensional closed manifold $K_j^{2n-2k-1}$, which is usually refered as the link of the stratum $\mathbf{X}_j^k$.

In other words, the pair $(\mathbb{D}_j^{2n-2k+2}, \mathbb{D}\cap V)$ is homeomorphic to the cone on $(S_j^{2n-2k+1}, K_j)$. Furthermore, we may assume that all the link pairs satisfy the Milnor fibration theorem.

Recall that we assume $H$ to be a generic hyperplane relative to $V$, so $H$ intersects all strata of $V$ transversally. The transversality condition of $H$ implies that the following maps remain locally trivial:

$$\xymatrix{
S_j^{2n-2k+1}\setminus K_j \simeq \mathbb{D}_j^{2n-2k+2} \setminus \mathbb{D}\cap V \ar[r] & \mathbf{Y}_j^k \setminus H := Y_j^k \ar[d]\\
& \mathbf{X}_j^k \setminus H := X_j^k}$$

Therefore, $T= \cup_{k,j}Y_j^k$ with each $Y_j^k$ serving as a total space of a fibration. If the intersection $Y_{j_1}^{k_1} \cap Y_{j_2}^{k_2}$ is non-empty and $k_2 \geq k_1$, then this intersection is actually the total space of a fibration over a submanifold of $X_{j_2}^{k_2}$ with same fiber as $Y_{j_2}^{k_2} \rightarrow X_{j_2}^{k_2}$, which is $S^{2n-2k_2+1}\setminus K^{2n-2k_2-1}$.

The local twisted Alexander modules are defined as we defined them for $T$ via composition on fundamental groups. All local meridians are mapped to the meridian around $V_1$ by inclusion. Notice that $$H_l^{\rho}(S^{2n-2k+1}\setminus K^{2n-2k-1}; \Ft) \cong H_l(F;\mathbb{V}_{\rho}),$$ where $F$ is the corresponding Milnor fiber. $F$ has a finite CW-complex structure of dimension $n-k$ (\cite{MT}, \cite{CS}). If we further assume $\mathbb{V}=\mathbb{F}=\Q$ and $\rho$ is trivial, then the local twisted Alexander module is the ordinary one over $\Q$. 

\begin{definition}
The pair $(\varepsilon,\rho)$ is said to be of locally finite type on $V_i$ if for every local link pair on $V_i$, its local Alexander modules $\TA_*(S^{2n-2k+1}\setminus K^{2n-2k-1}; \Ft)$ are torsion. If this is the case, the order of the local modules are denoted by $\xi_{k,l}^{\varepsilon,\rho}(t)$ and are refered as `local' polynomials.
\end{definition}

\begin{remark}
$(lk\#,\rho)$ is of locally finite type on all components, for all linear representations $\rho$.
\end{remark}

\subsection{Divisibility results, comparisons, and applications}

In this section, we give the statement of then main theorem of this paper and compare it with previous existing divisibility results.

\begin{thm}\label{main}
Let $V$ be a reduced hypersurface in $\CP$ with generic hyperplane $H$ at infinity. Assume that the pair $(\varepsilon,\rho)$ is of locally finite type on some irreducible component of $V$, say $V_1$. Then the $i$-th twisted Alexander module $\TA_i(U;\Ft)$ is torsion for $0\leq i \leq n$, with order denoted by $\Delta_i^{\varepsilon,\rho}(t)$.

For $0 \leq i \leq n$, the prime factors of the $i$-th twisted Alexander polynomial $\Delta_i^{\varepsilon,\rho}(t)$ of $V$ are among the prime factors of the local twisted Alexander polynomials $\xi_{k,l}^{\varepsilon,\rho}(t)$ associated to strata of $V_1$, with:
\begin{itemize}
\item $n-i \leq k \leq n$
\item $3n-3k-2i \leq l \leq n-k.$
\end{itemize}
\end{thm} 

\begin{cor}
The $k$-strata of $V$ will only contribute to $\Delta_n^{\varepsilon,\rho}(t),..., \Delta_{n-i}^{\varepsilon,\rho}(t)$. In particular, the 0-strata of $V$ only contribute to $\Delta_n^{\varepsilon,\rho}(t)$, the 1-strata contribute only to $\Delta_n^{\varepsilon,\rho}(t)$ and $\Delta_{n-1}^{\varepsilon,\rho}(t)$.
\end{cor}

Let us compare the theorem with one of the main theorem in \cite{MT} for ordinary Alexander polynomials: 

\begin{thm}\cite{MT}
Let $V$ be a reduced hypersurface in $\CP$ with generic hyperplane $H$ at infinity. Fix an arbitrary irreducible component of $V$, say $V_1$. Then for a fixed $i$, $1 \leq i \leq n$, the prime factors of the global Alexander polynomial $\delta_i(t)$ of $V$, which is the order of the module $H_i(U;\Qt)$ are among the prime factors of local polynomials $\xi_{k,l}(t)$, which are orders of local Alexander modules $H_l (S^{2n-2k+1}\setminus K^{2n-2k-1};\Qt)$ of link pairs associated to components of strata in $V_1$, such that: $$n-i\leq k\leq n$$ and $$2n-2k-i\leq l \leq n-k.$$
\end{thm} 

\begin{remark}
The ranges that Maxim obtained for ordinary Alexander polynomials are sharper then we have for the twisted case when $n\geq 4$, and are equivalent for lower $n$'s.
\end{remark}

\begin{example}
Suppose the Whitney stratification of $(\CP,V)$ has only three elements $$S=Sing(V)\subset V \subset \CP.$$ In particular, Sing$(V)$ is smooth. If futhermore Sing$(V)$ is $k$-dimensional, then $\Delta_{n-k}^{\varepsilon,\rho}| \xi^{\varepsilon,\rho}_{n-k}$. Such an example can be obtained by considering two smooth hypersurfaces intersecting transversely.
\end{example}

In \cite{AC} and \cite{HG}, Libgober proved a divisibility result for Alexander polynomials of curves and hypersurfaces with only isolated singularities. Theorem 5.4 implies similar results.

\begin{thm}
Let $V$ be a reduced hypersurface in $\CP$ with generic hyperplane at infinity and having only isolated singularities. Then the Alexander modules $H_i(U;\Qt)=0$ for $1 \leq i \leq n-1$, and the $n$-th Alexander module is torsion with order denoted by $\delta_n(t)$. Up to powers of $(t-1)$, $$\delta_n(t) \mid \prod_s \xi_{0,n}(t),$$ where $s$ runs over all singular points.
\end{thm}

\begin{remark}
Libgober's divisibility result is also true for hypersurface with isolated singularities, including at infinity. Note that the Alexander module defined with a non-generic hyperplane may not be torsion in the case of non-isolated singularities (\cite{DN}).
\end{remark}

\section{Proof of the theorem}

\subsection{Twisted Alexander modules as sheaf cohomologies}

The main goal of this section is to write the twisted Alexander modules of $V$ as sheaf cohomologies. For the untwisted case, one can see \cite{MT}.

\begin{prop}\cite{DB2}
If $X$ is a topological space with $G= \pi_1(X)$, and $\mathcal{L}$ is a local system defined on $X$ by $\rho$ with stalk $M$, then for $H\lhd G$ such that $H \subseteq Ker\rho$, and $G'=G/H$ and $X_H$ is the covering space of $X$ associated to $H$, we have the following:
$$C_*(X,\mathcal{L}) = C_*(X_H) \otimes_{\Z G'} M$$ and
$$C^*(X,\mathcal{L}) = Hom_{\Z G'}^*(C_*(X_H),M).$$
\end{prop}

Define a local system $\mathcal{L}$ on $U$ with stalk $\Ft^r$, where $r = dim_{\mathbb{F}}\mathbb{V}$, by:
$$\pi_1(U) \rightarrow Aut(\Ft^r)$$ with $$[\alpha] \mapsto t^{\varepsilon(\alpha)} \cdot \rho (\alpha).$$ 

By above proposition, $\TA_i(U;\Ft) = H_i(U;\mathcal{L})$. Now let $\mathcal{L}^{\vee}$ be the dual local system $\mathcal{H}om(\mathcal{L}, \mathcal{F}_U)$, where $\mathcal{F}_U$ is the constant sheaf. Assume that $\mathbb{F}$ is a subfield of $\C$ closed under conjugation. Then by \cite{KL} or \cite{DB2}, 
\begin{align*}
C^*(U,\mathcal{L}^{\vee})& =Hom_{\mathbb{F}[\pi']} (C_*(\tilde{U}), \Ft \otimes \mathbb{V}^{\vee}) \\
                & = Hom_{\Ft} (\TC_*(U; \Ft), \Ft)  \\
                & = Hom (C_*(U;\mathcal{L}),\Ft)
\end{align*}

The two groups on different sides have conjugate $\Ft$-module structure. By denoting the (co)homology group of $C^*(U,\mathcal{L}^{\vee})$ with conjugate module structure as $\overline{H}^*(U,\Ft)$, we have $$\overline{H}^*(U,\Ft) \cong H^*(\text{Hom}_{\Ft}(C_*(U,\Ft),\Ft)),$$ as $\Ft$-modules. By the universal coefficient theorem, $$\overline{H}^q(U,\Ft) \cong \text{Hom}_{\Ft}(\TA_q(U,\Ft),\Ft) \oplus \text{Ext}_{\Ft}(\TA_{q-1}(U,\Ft), \Ft).$$

As a corollary, if the twisted Alexander modules of $V$ are torsion modules (except the top one), then $\overline{H}^i(U,\Ft)$ is torsion for $i\leq n$ and $\Delta_i^{\varepsilon,\rho}(t)$ is equal to the order of $\overline{H}^{i+1}(U,\Ft)$ up to units.

Note the cohomology version also applies to any submanifold of $U$, because the modules are defined by composing $\varepsilon$ and $\rho$ with inclusion on fundamental group and is the cohomology with $i^*\mathcal{L}^{\vee}$.

\subsection{Proof of the theorem}

The proof of main theorem \ref{main} follows from an application of the Lefschetz hyperplane theorem and the following lemma.

\begin{lem}
Let $V$ be a reduced hypersurface in $\CP$ with a generic hyperplane $H$ at infinity (in the stratified sense). Assume that the pair $(\varepsilon,\rho)$ is of locally finite type on some irreducible component of $V$, say $V_1$. Then the $i$-th twisted Alexander module of $V$ is torsion for $0\leq i \leq n$.

Moreover, for $0 \leq i \leq n$, the prime factors of the $i$-th twisted Alexander polynomial $\Delta_i^{\varepsilon,\rho}(t)$ of $V$ are among the prime factors of local twisted Alexander polynomials $\xi_{k,l}^{\varepsilon,\rho}(t)$ belonging to $V_1$ with:
\begin{itemize}
\item $0 \leq k \leq n$
\item $i-3k \leq l \leq n-k$ 
\end{itemize}
\end{lem}

To simplify notations, we denote $\TA_{\bullet}(-;\Ft)$ by $\TA_{\bullet}(-)$; and $\overline{H}^{\bullet}(-;\Ft)$ by $\overline{H}^{\bullet}(-)$. Recall that $T=T(V_1)\setminus V_1\cup H$. A spectral sequence argument (cf \cite{EM}) will show that $\TA_i(T)$, $0\leq i\leq n$, are torsion, then so are $\TA_i(U)$.

Consider the Mayer-Vietoris spectral sequence for the twisted Alexander cohomologies \cite{TV} P.815:
\begin{equation}\label{E1}
E_1^{p,q}:\oplus_{\alpha} \overline{H}^q(\cap_{j=0}^p Y_{t_{\alpha, j}}^{k_{\alpha, j}})\Rightarrow \overline{H}^{p+q}(T)
\end{equation}
with $k_{\alpha,1}<k_{\alpha,2}<...<k_{\alpha,p}$.

Fix $P$, $Q$, and $\alpha$. By section 5, $\cap_{j=0}^P Y_{t_{\alpha, j}}^{k_{\alpha, j}}$ fibers over a subset of $X_{t_{\alpha, P}}^{k_{\alpha, P}}$. Denote that subset by $Z_{\alpha}^{P,Q}$ and consider the Leray spectral sequence for this locally trivial fiber bundle $\pi_{\alpha,P}$ (\cite{TV} P.814):
\begin{equation}\label{E2}
E_2^{p,q}:H^p(Z_{\alpha}^{P,Q};R^q(\pi_{\alpha,P})_* \mathcal{L}^{\vee}) \Rightarrow \overline{H}^{p+q} (\cap_{j=0}^P Y_{t_{\alpha, j}}^{k_{\alpha, j}}).
\end{equation}

By P.143 in \cite{MT}, a $E_2^{p,q}$ term in (\ref{E2}) is a finite direct sum of modules of the form $(R^q\pi_* \mathcal{L}^{\vee})_{x(\sigma)}$, where $x(\sigma)$ is the bary center of a $p$-simplex $\sigma$ of $Z_{\alpha}^{P,Q}$. And $(R^q(\pi_{\alpha,P})_* \mathcal{L}^{\vee})_{x(\sigma)} \cong H^q(\pi^{-1}(x(\sigma)); \mathcal{L}^{\vee}) =\overline{H}^q(S^{2n-2k_{\alpha, P}+1}\setminus K^{2n-2k_{\alpha, P}-1}).$

Since we assume that the pair $(\varepsilon,\rho)$ is of locally finite type, the link pairs are of finite type. Then by universal coefficient theorem, the twisted local cohomologies are torsion $\Ft$ modules and finite dimensional $\mathbb{F}$-vector spaces. Therefore, each $E_{\infty}^{p,q}$ term in (\ref{E2}) is torsion because it is a quotient of a submodule of $E_2^{p,q}$. Now, each $E_2^{p,q}$ in (\ref{E1}) is a direct sum of some $E_{\infty}^*$'s in (\ref{E2}), hence each of them is again torsion, and so is $\overline{H}^*(T)$. By UCT again, the twisted Alexander modules of the hypersurface for $0\leq i \leq n$ are torsion.

An algorithm is provided to show that the prime factors of $\Delta_{i-1}^{\varepsilon,\rho}(t)$ are among those of $\xi_{k,l-1}^{\varepsilon,\rho}(t)$, with $1 \leq i-1 \leq n$, ie, $2\leq i \leq n+1$.
\begin{itemize}
\item $\Delta_{i-1}^{\varepsilon,\rho}(t)$ divides the order of $\overline{H}^i(T)$, by UCT.
\item The order of $\overline{H}^i(T)$ equals to $\prod_{p+q=i}$(order of $E_{\infty}^{p,q}$) in (\ref{E1}).
\item Fix a pair of $P$ and $Q$ such that $P+Q=i$, then the order of $E_{\infty}^{P,Q}$ divides that of $E_1^{P,Q}$ in \ref{E1}.
\item The order of $E_2^{P,Q}$ equals $\prod_{\alpha}$(order of $\overline{H}^Q(\cap_{j=0}^P Y_{t_{\alpha, j}}^{k_{\alpha, j}}$)).
\item Fix an $\alpha$, the order of $\overline{H}^Q(\cap_{j=0}^P Y_{t_{\alpha, j}}^{k_{\alpha, j}}$) equals $\prod_{m+l=Q}$(order of $E_{\infty}^{m,l}$) in (\ref{E2}).
\item The order of $E_{\infty}^{Q-l,l}$ divides that of $E_2^{Q-l,l}$ in (\ref{E2}).
\item Each $E_2^{Q-l,l}$ in (\ref{E2}) is of the form  $$ \oplus_N \overline{H}^l(S^{2n-2k_{\alpha,P}+1}\setminus K^{2n-2k_{\alpha,P}-1}) = \oplus_N H_{l-1}(S^{2n-2k_{\alpha,P}+1}\setminus K^{2n-2k_{\alpha,P}-1}).$$
\end{itemize}

 Recall that order of $H_{l-1}(S^{2n-2k_{\alpha,P}+1}\setminus K^{2n-2k_{\alpha,P}-1})$ is denoted by $\xi_{k_{\alpha,P},l-1}^{\varepsilon,\rho}(t)$.

Keep in mind that $P+1$ equals the number of intersection of $Y$'s,  and hence $0\leq k_{\alpha,P}$.

To get the above claimed bounds for $l$, we determine which of the $E_2^{p,q}$ term in (\ref{E2}) vanishes.

Notice that there is a Milnor fibration for each link pair $(S^{2n-2k_{\alpha,P}+1}, K^{2n-2k_{\alpha,P}-1})$, $$F_{k_{\alpha,P}}\hookrightarrow S\setminus K \rightarrow S^1$$ where $F$ is the Milnor fiber. By \cite{SP}, Theorem 5.1, $F_{k_{\alpha,P}}$ has the homotopy type of a CW-complex of dimension $n-k_{\alpha, P}$. Thus, we only look at those $l-1$ in the range $$0\leq l-1 \leq n-k_{\alpha, P}$$

Furthermore, $H^m(Z_{\alpha}^{P,Q};-)$ is nontrivial only when $0\leq m\leq 2k_{\alpha,P}$, because $Z_{\alpha}^{P,Q} \subset X^{k_{\alpha,P}}$ is a complex $k_{\alpha,P}$-dimensional manifold.

As a result, we have the following:

\begin{itemize}
\item $P+m+l=i$
\item $0 \leq i-P-l \leq 2k_{\alpha,P}$
\item $l-1 \geq (i-1)-2k_{\alpha,P}-P\geq (i-1)-3k_{\alpha,P}$
\end{itemize}

In conclusion, we have for $0\leq i-1 \leq n$, $l-1$ is bounded for all $\alpha$, by $$(i-1)-3k_{\alpha,P}\leq l-1 \leq n-k_{\alpha,P}.$$

By renaming $i-1$ by $i$ and $l-1$ by $l$, then for $0\leq i \leq n$, we have $$i-3k_{\alpha,P} \leq l \leq n-k_{\alpha,P}.$$ 

Now we apply the Lefschetz hyperplane theorem (cf.\cite{MT}).

Let $1\leq i = n-j \leq n$ be fixed. Let $L \cong \mathbb{CP}^{n-j+1}$ be a generic codimension $j$ linear subspace of $\CP$. By transversality, $W= V\cap L$ is a $(n-j)$ dimensional, degree $d$, reduced hypersurface in $L$, and is transversal to the hyperplane at infinity $H\cap L$. Also, the pair $(L,W)$ has Whitney stratification induced from the pair $(\CP,V)$ with strata of the form $S\cap L$ where $S$ is a stratum of the pair $(\CP,V)$.

Applying Lefschetz hyperplane theorem to $U$ and $L$, $U\cap L \rightarrow U$ is an $(n-j+1)$-equivalence. The homotopy type of $U$ is obtained from $U\cap L$ by adding cells of dimension  greater than $n-j+1$.

Therefore, we have isomorphisms as $\Ft$-modules, for $i\leq n-j$, $$\TA_i(U\cap L) \rightarrow \TA_i(U).$$

Thus, $\Delta_{n-k}^{\varepsilon,\rho,W}(t) = \Delta_{n-k}^{\varepsilon,\rho,V}(t)$. Note that $\Delta_{n-j}^{\varepsilon,\rho,W}(t)$ is the top polynomial of $W$ as a hypersurface in $L \cong \mathbb{CP}^{n-j+1}$. Then by lemma 6.1, the prime factors of $\Delta_{n-j}^{\varepsilon,\rho,W}(t)$ are among those of the local polynomials $\xi_{r,l}^{\varepsilon,\rho}$ of link pair of the strata $S\cap L \subset V_1 \cap L$ with:
$$0 \leq r = \text{dim}(S\cap L) \leq n-j$$ and 
\begin{equation}\label{E3}
(n-j)-3r \leq l \leq (n-j)-r.
\end{equation}

The link pair of $S\cap L$ is the same as that of $S$ in $(\CP,V)$. Reindexing $r=k-j$ with $k = \text{dim}(S)$, $k = r+j = r+n-i \geq n-i.$

By substituting $j=n-i$ and $r=k-j=k+i-n$, the right-hand side of \ref{E3} becomes $n-k$ and the left-hand is calculated as follows:

$$n-(n-i)-3(k+i-n)=n-n+i-3k-3i+3n = 3n-3k-2i.$$

\section{Obstructions of twisted Alexander polynomials using Perverse sheaves}

In this section, we prove that twisted Alexander polynomials of hypersurface complements have the same obstruction as $\Delta_1^{lk\#,\rho}$ for curve complements. The first idea is to use section 4.3 and analyze $H_i^{\rho}(T_H,\Ft)$ using a Mayer-Vietoris spectral sequence. The second idea is to augment a proof appearing in p.556-557 in \cite{RF}, which is short and neat. Both approaches require calculations of local (co)homology on $H$.

Let $\gamma_{\infty}$ be a loop at infinity of $U$,a meridian around $H$ with homology class $\nu_{\infty}$ satisfying $$d_1\nu_1+d_2\nu_2...+d_r\nu_r = -\nu_{\infty},$$ where $d_i$ is the degree of the component $V_i$.

Note that $\gamma_{\infty}$ is analogous to the $x_0^{-1}$ mentioned in the section on curve complements.

\begin{thm}
Let $\rho$ be a representation. If $\lambda\in \C^*$ and $\lambda^d$ is not an eigenvalue of $\rho(\gamma_{\infty})$, then $\lambda$ cannot be a root of $\Delta_i^{\rho}$, for $i\leq n$.
\end{thm}

\begin{proof}
Let $\mathcal{L}_{\lambda}$ be the local system on $U$ defined by the map $\pi_1(U) \rightarrow H_1(U) \rightarrow \C^*$ sending $\nu_i$ to $\lambda$ and hence $\gamma_{\infty}$ to $\lambda^{-d}$. By the twisted Milnor long exact sequence of section 4.4, it is enough to show that for $i\neq n+1$, $$H_i(U,\mathcal{L}_{\lambda}\otimes \mathbb{V}_{\rho})=0.$$

Equivalently, we show that for $i\neq n+1$, $H^i(U,\mathcal{L}_{\lambda}\otimes \mathbb{V}_{\rho})=0.$

Let $i:U \hookrightarrow \C^{n+1}$ and $j: \C^{n+1} \hookrightarrow \CP$. Note that $\mathcal{F} = Ri_*(\mathcal{L}_{\lambda}\otimes \mathbb{V}_{\rho}[n+1]) \in \text{Perv}(U).$

Define $\mathcal{G}$ by the distinguished triangle $Rj_!\mathcal{F}\rightarrow Rj_*\mathcal{F}\rightarrow\mathcal{G}.$

By P.214 in \cite{DB2}, this theorem will follow if $\mathcal{G}=0$. So, we aim to show that the local cohomology groups of $Rj_*\mathcal{F}$ vanish at all points in $H$. The local link pair of a point on $H\setminus V$ is $(S^1,\phi)$. The associated chain complex for local cohomology is $$0\rightarrow \C^r \xrightarrow{\lambda^{-d}M-I_r} \C^r \rightarrow 0.$$

By hypothesis, $\lambda^{-d}M-I_r$ has full rank and hence the cohomologies vanish.

For points in $H\cap V$, by transversality, the local link complements are homotopy equivalent to the product of $S^1 $ and some link complements $S\setminus L$ in $V$. Therefore, by the K$\ddot{u}$nneth theorem, the local cohomology is the tensor product of the local cohomology of $S\setminus L$ and that of $S^1$. So they all vanish by the previous argument.

\end{proof}


\begin{bibdiv}
\begin{biblist}


\bib{CS}{article}{
title={Singular spaces, Characteristic classes, and Intersection homology},
author={Cappell, S.},
author={Shaneson, J.},
journal={Annals of Mathematics},
volume={134},
date={1991},
pages={325-374}
}

\bib{TC}{article}{
title={Twisted Alexander Polynomials of Plane Algebraic Curves},
author={J.I. Cogolludo Agustin},
author={V. Florens},
journal={J. Lond. Math. Soc. (2) no. 1},
volume={76},
date={2007},
pages={105-121}
}


\bib{BM}{article}{
title={The boundary manifold of a complex line arrangement},
author={Daniel Cohen},
author={Alexander Suciu},
date={2008},
journal={Geometry \& Topology Monographs},
volume={13},
pages={105-146}
}

\bib{TV}{book}{
title={Toric Varieties},
author={Cox, David A.},
author={Little, John B.},
author={Schenck, Henry K.},
date={2011},
series={Graduate Studies in Mathematics},
publisher={American Mathematical Society}
}



\bib{FB}{book}{
title={Introduction to knot theory},
author={Crowell, Richard},
author={Fox, Ralph},
date={1963},
series={Graduale Texts In Mathematics},
volume={57},
publisher={Ginn}
}

\bib{RF}{article}{
title={Regular Functions Transversal at Infinity},
author={Dimca, Alexandru},
author={Libgober, Anatoly},
journal={Tohoku Math. J.},
volume={58},
date={2006},
pages={549-564}
}

\bib{DB2}{book}{
title={Sheaves in Topology},
author={Dimca, Alexandru},
date={2004},
series={University Text},
publisher={Springer}
}

\bib{DB}{book}{
title={Singularities and Topology of Hypersurfaces},
author={Dimca, Alexandru},
date={1992},
series={University Text},
publisher={SpringerVerlag}
}



\bib{DN}{article}{
title={Hypersurface complements, Alexander modules and monodromy},
author={Dimca, Alexandru},
author={Nemethi, A.},
journal={Contemp.Math.},
volume={354},
date={2004},
publisher={Amer.Math.Soc}
}

\bib{DP}{article}{
title={Hypersurface complements, Milnor fibers and higher homotopy groups of arrangements},
author={Dimca, Alexandru},
author={Papadima, Stefan},
journal={Ann. of Math.},
volume={158},
date={2003},
pages={473--507}
}


\bib{FC}{article}{
title={Free differential calculus II. The isomorphism problem of groups},
author={Fox,R.},
journal={Annals of Mathematics},
volume={59(2)},
date={1954},
pages={196-210}
}

\bib{AH}{book}{
title={Algebraic Topology},
author={Hatcher, Alan},
date={2001},
publisher={Cambridge University Press}
}

\bib{KL}{article}{
title={Twisted Alexander invariants, Reidemeister torsion, and Casson-Gordon invariants},
author={Kirk, P.},
author={Livingston, C.},
journal={Topology no.3},
volume={38},
date={1999},
pages={635-661}
}

\bib{TS}{article}{
title={Twisted Alexander polynomials and surjectivity of a group homomorphism},
author={Kitano, Teruaki},
author={Suzuki,Masaaki},
author={Wada, Masaaki},
journal={Algebraic and Geometric Topology},
volume={5},
date={Oct 6th 2005},
pages={1315-1324}
}

\bib{AC}{article}{
title={Alexander polynomial of plane algebraic curves and cyclic multiple planes},
author={Libgober, Anatoly},
journal={Duke Mathematical Journal},
volume={49(4)},
date={1982},
pages={833-851},
}

\bib{EM}{article}{
title={Eigenvalues for the Monodromy of the Milnor fibers of Arrangements},
author={Libgober, Anatoly},
journal={Trends in Singularities},
date={2002},
pages={141-150},
}

\bib{HG}{article}{
title={Homotopy Groups of the Complements to Singular hypersurfaces II},
author={Libgober, Anatoly},
journal={Annals of Math.},
volume={139},
date={1994},
pages={117-144}
}

\bib{NV}{article}{
title={Non vanishing loci of Hodge numbers of local systems},
author={Libgober, Anatoly},
journal={Manuscripta Mathematica},
volume={128(1)},
date={2008},
pages={1-31}
}

\bib{LT}{article}{
title={Nearby Cycles and Alexander Modules of Hypersurface Complements},
author={Liu, Yongqiang},
journal={arXiv:1405.2343}
}

\bib{LIN}{article}{
title={Representations of knot groups and twisted Alexander polynomial},
author={Lin, X.S},
journal={Acta Math.Sin},
date={2001}
}


\bib{MT}{article}{
title={Intersection Homology and Alexander Modules of Hypersurface complements},
author={Maxim, Laurentiu},
journal={Comment. Math. Helv.},
volume={81 no.1},
date={2005},
pages={123-155}
}

\bib{L2}{article}{
title={L2-Betti numbers of hypersurface complements},
author={Maxim, Laurentiu},
journal={Int. Math. Res. Not.}
volume={2014, No. 17},
pages={4665-4678}
}
\bib{SP}{book}{
title={Singular points of Complex Hypersurfaces},
author={Milnor, John},
date={1968},
series={Annuals of Mathematical Studies 61},
volume={50},
publisher={Princeton Univ. Press},
address={Princeton, NJ}
}

\bib{OKA}{article}{
title={A survey on Alexander polynomials of plane curves},
author={Oka, M.},
journal={S$\acute{e}$minaires Cong$\acute{e}$s},
date={2005},
volume={10},
pages={209-232}
}

\bib{AO}{article}{
title={On the fundamental group of the complement of certain plane curves},
author={Oka, M.},
journal={J. Math.Soc.Japan},
date={1978},
volume={30},
pages={579-597}
}


\bib{RB}{book}{
title={Knots and Links},
author={Rolfsen, Dale},
date={1976},
publisher={AMS Chelsea Pub.}
}


\bib{WD}{article}{
title={Twisted Alexander polynomial for finitely presentable group},
author={Wada, M.},
date={1994},
journal={Topology},
volume={33 no.2},
pages={241-256}
}


\end{biblist}
\end{bibdiv}
\end{document}